\documentclass[12pt]{amsart}
\usepackage{latexsym,amssymb,amsmath,hyperref,amsthm,amsfonts,
caption,subcaption,tikz,comment}
\usepackage{mdwlist,dirtytalk,times}
\usepackage[capitalise, noabbrev]{cleveref}
\usepackage{blkarray}
\usetikzlibrary{
intersections, arrows.meta,
automata,er,calc,
backgrounds,
mindmap,folding,
patterns,
decorations.markings,
fit,decorations,
shapes,matrix,
positioning,
shapes.geometric,
arrows,through, graphs, graphs.standard
}
\usepackage{amsrefs}

\numberwithin{equation}{section}
\usetikzlibrary{positioning}
\usetikzlibrary{arrows}

\newtheorem{theorem}{Theorem}[section]
\newtheorem{lemma}[theorem]{Lemma}

\newtheorem{corollary}[theorem]{Corollary}

\newtheorem{question}[theorem]{Question}

\theoremstyle{definition}
\newtheorem{definition}[theorem]{Definition} 
\newtheorem{remark}[theorem]{Remark}
\newtheorem{example}[theorem]{Example}

\newcommand{\K}{\mathbb{K}}

\DeclareMathOperator{\Ind}{Ind}
\DeclareMathOperator{\cha}{char}
\DeclareMathOperator{\Ann}{Ann}

\newcommand{\qand}{\quad \mbox{and} \quad}

\newcommand{\qfor}{\quad \mbox{for} \quad}

\newcommand{\qwhere}{\quad \mbox{where} \quad}

\newcommand{\bff}{\mathbf{f}}

\begin{document}
 
\title{The Weak Lefschetz property of whiskered graphs}

\author[S.~M.~Cooper]{Susan M. Cooper}
\address[S. M. Cooper]
{Department of Mathematics\\
University of Manitoba\\
520 Machray Hall\\
186 Dysart Road\\
Winnipeg, MB\\
Canada R3T 2N2}
\email{susan.cooper@umanitoba.ca}

\author[S. Faridi]{Sara Faridi}
\address[S. Faridi]
{Department of Mathematics \& Statistics,
Dalhousie University,
6316 Coburg Rd.,
PO BOX 15000,
Halifax, NS,
Canada B3H 4R2
}
\email{faridi@dal.ca}

\author[T. Holleben]{Thiago Holleben}
\address[T. Holleben]
{Department of Mathematics \& Statistics,
Dalhousie University,
6316 Coburg Rd.,
PO BOX 15000,
Halifax, NS,
Canada B3H 4R2}
\email{hollebenthiago@dal.ca}

\author[L. Nicklasson]{Lisa Nicklasson}
\address[L. Nicklasson]
{Division of Mathematics and Physics,
Mälardalen University,
Box 883, 721 23 Västerås, Sweden}
\email{lisa.nicklasson@mdu.se}

\author[A. Van Tuyl]{Adam Van Tuyl}
\address[A. Van Tuyl]
{Department of Mathematics and Statistics\\
McMaster University, Hamilton, ON, L8S 4L8}
\email{vantuyl@math.mcmaster.ca}

\keywords{Weak Lefschetz property, graded Artinian rings, whiskered graphs, pseudo-manifolds}
\subjclass[2020]{13E10, 13F20, 13F55, 05E45}

 
\begin{abstract} 
We consider Artinian level algebras arising from the whiskering of a graph. Employing a result by Dao-Nair we show that multiplication by a general linear form has maximal rank in degrees 1 and $n-1$ when the characteristic is not two, where $n$ is the number of vertices in the graph. Moreover, the multiplication is injective in degrees $<n/2$ when the characteristic is zero, following a proof by Hausel. Our result in the characteristic zero case
is optimal in the sense that there are whiskered graphs for which the multiplication maps in all intermediate
degrees $n/2,\ldots,n-2$ of the associated Artinian algebras fail to have maximal rank, and 
consequently, the weak Lefschetz property.
\end{abstract}
\maketitle


\section{Introduction}

A graded Artinian algebra $A=A_0 \oplus A_1 \oplus \dots 
\oplus A_t$ has the \emph{weak Lefschetz property (WLP)} 
if there is an $\ell \in A_1$ such that the multiplication 
maps $\times \ell: A_i \to A_{i+1}$ all have maximal rank, 
i.e., are injective or surjective. 
When studying the WLP it is natural to consider standard 
graded algebras, presented as $A=R/I$ where 
$R=\mathbb{K}[x_1,\ldots,x_n]$ is a polynomial ring over a 
field $\mathbb{K}$ and $I$ is a homogeneous ideal. As was 
first pointed out in \cite{MMN2011}, when $I$ is a monomial 
ideal, $A$ has the WLP if and only if the multiplication maps 
induced by $\ell=x_1+ \cdots +x_n$ have maximal rank. A 
recent contribution to the investigation of WLP 
for monomial ideals is work  of Dao and Nair 
\cite{DN2021} where the  Stanley-Reisner ideals $I_{\Delta}$ 
together with the squares  of the variables are considered over a field of characteristic zero. Their results, which are described in
detail in \cref{sec:background}, complement 
previous work by Michałek and Mir\'{o}-Roig \cite{MM2016}, 
and Migliore, Nagel and Schenck \cite{MNS2020}, both studying 
the WLP of quadratic and cubic monomial ideals using 
Togliatti systems.  Dao and Nair's work also complements work by
Cook, Migliore, Nagel and 
Zanello \cite{CMNZ2016} that
relates the WLP of monomial
Artinian algebras to problems in incidence geometry.

A quadratic monomial ideal defining an Artinian algebra can 
be interpreted as an edge ideal of a graph together with the 
squares of the variables. More precisely, let 
$G = (V,E)$ be a finite simple graph on the vertex set 
$V = \{x_1,\ldots,x_n\}$ and edge set $E$. The 
{\it edge ideal} $I(G) = \langle x_ix_j \mid \{x_i,x_j\} 
\in E \rangle$ defines an Artinian algebra 
$A(G) = R/(\langle x_1^2,\ldots,x_n^2 \rangle + I(G))$, called the Artinian
algebra of $G$.  We are interested in the following general question:

\begin{question}\label{question}
    For which graphs $G$ does $A(G)$ have the
    WLP?  If $A(G)$ does not have the WLP,
    in which degrees do the multiplication maps fail to 
    have maximal rank?
\end{question}

 The WLP of such algebras $A(G)$ have been studied in \cite{NT2022} and \cite{T2021}, where they classify the WLP for some special classes of graphs including paths, cycles, wheel graphs, and star graphs.  However, our understanding of \cref{question} is far from complete; we contribute to this question by considering 
 \emph{whiskered graphs} (see \cref{defn.whisker}). Artinian algebras defined by whiskered graphs are particularly interesting from the Lefschetz properties perspective, as they are level.

In \cite{TH2023}*{Theorem 58} Dao and Nair's result \cite{DN2021} is generalized to hold over any field of characteristic different from $2$.
Our contribution in this short note, summarized in Theorem \ref{maintheorem1} below, is an application of this result together with a result by Hausel \cite{Hausel} to whiskered graphs. Note that a whiskered graph always has $2n$ vertices and at least $n$ edges, for some number $n$, by construction. The socle degree of the corresponding Artinian algebra is $n$, meaning that the algebra is of the form $A_0 \oplus A_1 \oplus \cdots \oplus A_n$ with $A_n \ne \{0 \}$.

\begin{theorem}[\cref{cor.maincor,c:char}]\label{maintheorem1}
Let $G$ be a whiskered graph with $2n$ vertices and at least $n+1$ edges, and 
let $A=A(G)$ be the Artinian algebra of $G$ over a field $\K$ where $\cha(\K)\neq 2$. If $\ell$ is the sum of the variables of $A$, then $\times \ell: A_i \to A_{i+1}$ has maximal rank when $i=1$ or $n-1$. In the case $\cha(\K)=0$, the map 
$\times \ell$ is injective for all $i<n/2$.
\end{theorem}

Our proof of \cref{maintheorem1} relies on showing that
the independence complex of a whiskered
graph, a simplicial complex
whose faces correspond to the independent sets of the graph, is a pseudo-manifold with boundary.  This
result is of independent interest, and highlights some of
the nice combinatorial properties of whiskered graphs.
As we show later in the paper, \cref{maintheorem1} is the best statement we can hope to make for the entire class of whiskered graphs.
For further reading on edge ideals of whiskered graphs see
\cites{DE2009,FH2008,V1990,W2009}, and the generalizations given in \cites{BVT2013,CN2012,Fa2005}.

Our paper is structured as follows.  In the next
section we recall the relevant terminology and
results on Artinian algebras, simplicial complexes,
and graph theory.  In Section 3 we prove our main theorem
(see \cref{thm.mainthm}).  In Section 4 we provide some
illustrative examples to show that we cannot extend \cref{maintheorem1} to degrees $i=2$ or $i=n-2$.

\section{Background}\label{sec:background}
We recall the required results about Artinian rings and 
graph theory used to prove our main
results.  
For a general introduction to Artinian rings and the 
weak and strong Lefschetz properties we recommend \cite{MN2013}.

\subsection{Lefschetz properties of Artinian rings}

Throughout, we let $\K$ be a field.  A graded $\K$-algebra $A$ can be decomposed into a direct sum of vector spaces $A=\bigoplus_{i \ge 0}A_i$ such that $A_0 \cong \K$, and the multiplication satisfies $A_i A_j \subset A_{i+j}$. The algebra $A$ is {\it Artinian} precisely when the number of non-zero graded components is finite. 

\begin{definition}[Weak Lefschetz property]
An Artinian graded algebra $A= A_0 \oplus A_1  \oplus \dots \oplus A_t$ has the \emph{weak Lefschetz property (WLP)} if there exists a linear form $\ell \in A_1$ such that the map $\times \ell : A_i \to A_{i+1}$ given by $a \mapsto a \ell$ has maximal rank $\min\{\dim_{\K}A_i, \dim_{\K}A_{i+1}\}$ for each $i=0, \ldots, t-1$. If $A$ has the WLP with the linear form $\ell$, we say that $\ell$ is a \emph{weak Lefschetz element}. 
\end{definition}

Note that a weak Lefschetz element, if it exists, is in general not unique. In this paper we consider Artinian algebras of the form $\K[x_1, \ldots, x_n]/I$ where $I$ is a monomial ideal, and in this case it suffices to consider $\ell$ to be the sum of the variables. 

\begin{lemma}[\cites{MMN2011, LN2019}]\label{lem:linear_form}
An Artinian algebra $\K[x_1, \ldots, x_n]/I$, where $I$ is a monomial ideal, has the WLP if and only if $x_1 + \cdots + x_n$ is a weak Lefschetz element.  
\end{lemma}  

\cref{lem:linear_form} was first stated in \cite{MMN2011}*{Proposition 2.2} over an infinite field, and later in \cite{LN2019}*{Proposition 4.3} over an arbitrary field.

 Consider a graded Artinian algebra $A=A_0 \oplus \dots \oplus A_t$, assuming $A_t \ne \{0\}$. The \emph{socle} of $A$ is the subspace
 \[
 \{f\in A \ | \ af=0 \ \text{for any } a \in A_1 \}.
 \]
 It is clear that the socle contains the graded component $A_t$, and $A$ is called \emph{level} if the socle is precisely $A_t$. When $A$ is level we say that $t$ is the \emph{socle degree} of $A$. 

 \begin{lemma}[{\cite{MMN2011}*{Proposition 2.1}}]\label{lem:injsurj}
 Let $A=A_0 \oplus \dots \oplus A_t$ be a graded Artinian algebra, and $\ell \in A_1$. If the map $\times \ell :A_i \to A_{i+1}$ is surjective for some $i_0$, then the same holds for any $i>i_0$. Moreover, if $A$ is level, and   $\times \ell : A_i \to A_{i+1}$ is injective for some $i_0$, then the same holds for any $i<i_0$.
 \end{lemma}

When $A$ is a level algebra of characteristic 0 defined by a monomial ideal, an argument found in the proof of \cite{Hausel}*{Theorem 6.3} proves injectivity of several multiplication maps. For the reader's convenience we include a short proof here. 

\begin{lemma}\label{lemma:monomial_level_injective}
Let $A=\K[x_1, \ldots, x_n]/I$ be a level Artinian algebra of socle degree $t$, where $I$ is a monomial ideal, and ${\rm char}(\K) =0$. Let $\ell=x_1 + \dots + x_n$. Then the maps $\times \ell^{t-2k}:A_k \to A_{t-k}$ and $\times \ell: A_k \to A_{k+1}$ are both injective, for $k<t/2$. 
\end{lemma}

\begin{proof}
	We start by recalling the construction of a level algebra via Macaulay's inverse system. Let $R=\K[x_1, \ldots, x_n]$ and $S=\K[X_1, \ldots, X_n]$, and let the polynomials of $R$ act as differential operators on $S$ by taking $x_i=\frac{\partial}{\partial X_i}$. For $g \in S$ we let $\Ann(g)$ denote the ideal of $R$ consisting of all elements annihilating $g$ under this action. An Artinian algebra $R/I$ is level with socle degree $t$ if and only if 	$I= \Ann(g_1) \cap \dots \cap \Ann(g_s)$ for some $g_1, \ldots, g_s \in S_t$. For a proof of this fact see \cite{Lefschetz_book}*{Proposition 2.74}.
	
	Now assume $I$ is a monomial ideal such that $R/I$ is Artinian and level of socle degree $t$. Then $I= \Ann(m_1) \cap \dots \cap \Ann(m_s)$ for some monomials $m_1, \ldots, m_s \in S_t$. Indeed, the monomials $m_1, \ldots, m_s$ can be obtained by taking all monomials in $R_t$ not included in $I$, and substituting $x_i \mapsto X_i$.   
 Let $\ell=x_1 + \dots + x_n$, and let $f$ be an element in the kernel of the map $\times \ell^{t-2k}:A_k \to A_{t-k}$, for some $k<t/2$. Lifting $f$ to $R$ we have $f \in R_k$ such that $\ell^{t-2k}f \in I$. By the decomposition of $I$ we also have $\ell^{t-2k}f \in \Ann(m_i)$, for each $i=1, \ldots, s$. Taking $A^{(i)}=R/\Ann(m_i)$, this implies that $f$ is in the kernel of $\times \ell^{t-2k}:A^{(i)}_k \to A^{(i)}_{t-k}$. The algebra $A^{(i)}$ is a monomial complete intersection. It is well known that such algebras have the strong Lefschetz property, the first proof was given in \cite{S1980}. In particular, this says that the maps $\times \ell^{t-2k}:A^{(i)}_k \to A^{(i)}_{t-k}$ are bijective for all $k<t/2$. Hence $f \in \Ann(m_i)$ for each $i$, or in other words $f \in I$,  and we conclude that   the map $\times \ell^{t-2k}:A_k \to A_{t-k}$ is injective. Then the same holds for $\times \ell: A_k \to A_{k+1}$, as an element in the kernel of this map would also be in the kernel of  $\times \ell^{t-2k}:A_k \to A_{t-k}$. 
\end{proof}

Recall that a {\it simplicial complex} $\Delta$ on
a vertex set $V = \{x_1,\ldots,x_n\}$ is a set of subsets of $V$ that satisfies
the conditions: (1) if $F \in \Delta$ and
$G \subseteq F$, then $G \in \Delta$, and (2)
$\{x_i\} \in \Delta$ for all $i$.  An 
element $F \in \Delta$ is a {\it face}.  The 
maximal elements of $\Delta$ (with respect
to inclusion) are called the {\it facets} of
$\Delta$.  Given a face $F \in \Delta$, the
{\it dimension} of $F$ is $\dim(F) = |F|-1$.  
By convention, $\dim(\emptyset) = -1$.  
We sometimes call a face $F$ a {\it $t$-face}
if $\dim(F) =t$.
The {\it dimension of $\Delta$} is defined
to be $\dim(\Delta) = \max\{\dim(F) \mid F \in \Delta\}$.  A simplicial complex is {\it pure} if
all its facets have the same dimension.
The $\bff$-vector of a $d$-dimensional simplicial complex $\Delta$ is the vector of integers
$$    \bff(\Delta)=(f_0,\ldots,f_d) \qwhere f_i=\mbox{ number of $i$-dimensional faces of }\Delta.
$$

We are interested in the following class of simplicial
complexes.

\begin{definition}[Pseudo-manifold]\label{defn.pseudo}
A $d$-dimensional simplicial complex $\Delta$ is
a {\it pseudo-manifold} if the following conditions
hold:
\begin{enumerate}  
    \item $\Delta$ is pure;
    \item every face of dimension $(d-1)$ of
    $\Delta$ is contained in at most two
    facets of $\Delta$; and 
    \item for every two facets $F,F' \in \Delta$,
    there exists a sequence of facets $F=G_0,G_1,\ldots,G_t=F'$ such that 
    $\dim (G_i \cap G_{i+1}) = d-1$ for 
    all $i = 0,\ldots,t-1$.
\end{enumerate}
Additionally, a pseudo-manifold has a {\it boundary} if
there exists at least one face of dimension $(d-1)$
of $\Delta$ that belongs to exactly one facet of
$\Delta$.
\end{definition}

Given a simplicial complex $\Delta$, the
Stanley-Reisner ideal of $\Delta$ is the 
square-free monomial ideal 
$$I_\Delta = \langle x_{i_1}\cdots x_{i_r} \mid
\{x_{i_1},\ldots,x_{i_r}\} \not\in \Delta \rangle.$$
We construct a graded Artinian algebra from $\Delta$ 
as follows:
\begin{equation}\label{e:ad}
A(\Delta) = \mathbb{K}[x_1,\ldots,x_n]/
(\langle x_1^2,\ldots, x_n^2 \rangle + I_\Delta).
\end{equation}
From the definitions of the Stanley-Reisner ideal and the $\bff$-vector of $\Delta$, it follows that when $A=A(\Delta)$, then $$\dim_{\K}A_i=f_{i-1} \qfor i=1,\ldots,\dim(\Delta)+1.$$

The following theorem now links together a number of
concepts defined above.  In the statement below, the 
{\it $1$-skeleton} of $\Delta$ is the simplicial complex
consisting of all the faces of $\Delta$ of dimension
at most one, which can be considered a graph.  A graph
is {\it bipartite} if it contains no cycles of odd length.
Also in the statement below is the dual graph of $\Delta$.  We will not need to make use of the dual graph in this paper and thus do not define the term here. 
 
\begin{theorem}[{Dao--Nair~\cite{DN2021}*{Theorems 1.1, 1.2}}]\label{thm.daonair}
Let $\Delta$ be a simplicial complex with $1$-skeleton $G$, let 
$A = A(\Delta)$ be the Artinian $\K$-algebra defined in \eqref{e:ad}, where $\cha(\K)=0$, and let $\ell=x_1 + \dots + x_n$. 
\begin{itemize}
\item The map $\times \ell: A_1 \to A_2$ is injective if and only if $f_0 \le f_1$ (or equivalently $\dim_{\K}A_1 \le \dim_{\K}A_2$) and $G$ has no bipartite connected components. 

\item If $\Delta$ is a $d$-dimensional pseudo-manifold, then $\times \ell: A_d \to A_{d+1}$ has maximal rank if and only if 
  \begin{itemize}
     \item $\Delta$ has boundary, or
     \item $\Delta$ has no boundary, and the dual graph of $\Delta$ is not bipartite. 
  \end{itemize}
\end{itemize}

\end{theorem}


\subsection{Graph theory background}
Let $G = (V,E)$ denote a finite simple graph, where
$V = \{x_1,\ldots,x_n\}$ denotes the set of {\it vertices}
of $G$, and $E$ denotes the set of {\it edges} of
$G$. We will write $V(G)$, respectively $E(G)$, if we 
wish to highlight that the vertices, respectively edges, belong
to the graph $G$. Note that a graph is a 1-dimensional simplicial complex.

\begin{definition}[Independent set, complex]
For a graph $G = (V,E)$, a subset $W \subseteq V$
is called an {\it independent set}
if for all $e \in E$, at least one vertex of $e$ is not in $W$.  The
{\it independence complex} of $G$ is the simplicial complex 
$$
\Ind(G) = 
\{W \mid \mbox{$W \subseteq V$ is  an independent set of $G$}\}.
$$
\end{definition}

It is straightforward to verify that $\Ind(G)$
satisfies the definition of a simplicial complex.  
The facets of $\Ind(G)$ correspond to the
{\it maximal independent sets} of the graph $G$.
By the Stanley-Reisner correspondence,
it can be shown that 
$I_{\Ind(G)} = I(G)$,
that is, the square-free monomial  associated with the
simplicial complex $\Ind(G)$ is
 precisely  the edge ideal of $G$.  Instead of writing
$A(\Ind(G))$ for the graded Artinian algebra
of \cref{e:ad}, we will abuse
notation and simply write $A(G)$ (thus
agreeing with our notation in the 
introduction), and call $A(G)$ the
{\it Artinian algebra of $G$}.

We now turn our attention to the main class
of graphs we wish to study.

\begin{definition}[Whiskering]\label{defn.whisker}
Given a graph $G$ with vertex set $V = \{x_1,\ldots,x_n\}$, the 
{\it whiskered graph of $G$}, denoted $w(G)$, is
the graph on the vertex set $V(w(G)) = \{x_1,\ldots,x_n,
y_1,\ldots,y_n\}$ and edge set
$$E(w(G)) = E(G) \cup \{\{x_i,y_i\} \mid i = 1,\ldots,n\}.$$
We say a graph is {\it whiskered} if it is the whiskering of some graph $G$. 
\end{definition}

Informally, we ``whisker'' a graph $G$ by adding a new 
vertex $y_i$ for each vertex $x_i$ in $G$, and join these two vertices together with an edge. An example of a graph $G$ and its whiskered graph
$w(G)$ is given in \cref{fig:whisker}.
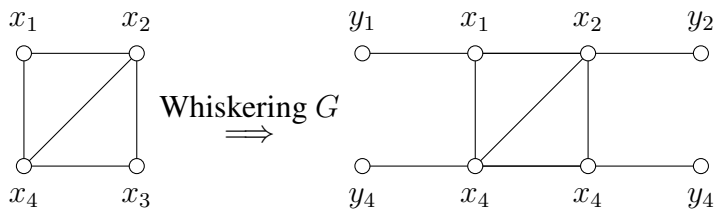
\begin{figure}[h]
\begin{tikzpicture}[scale=0.5]
\draw (6,0) -- (6,3) -- (9,3) -- (9,0) -- (6,0)-- (9,3);

\draw (18,0) -- (18,3) -- (21,3) -- (21,0) -- (18,0) -- (21,3);
\draw (15,0)-- (24,0);
\draw (15,3)-- (24,3);

\node at (12,1.5) {Whiskering $G$};
\node at (12,.8) {$\Longrightarrow$};

\fill[fill=white,draw=black] (6,0) circle (.2)
node[label=below:$x_4$] {};
\fill[fill=white,draw=black] (9,0) circle (.2)
node[label=below:$x_3$] {};
\fill[fill=white,draw=black] (6,3) circle (.2)
node[label=above:$x_1$] {};
\fill[fill=white,draw=black] (9,3) circle (.2)
node[label=above:$x_2$] {};

\fill[fill=white,draw=black] (15,0) circle (.2)
node[label=below:$y_4$] {};
\fill[fill=white,draw=black] (15,3) circle (.2)
node[label=above:$y_1$] {};
\fill[fill=white,draw=black] (18,0) circle (.2)
node[label=below:{$x_4$}] {};
\fill[fill=white,draw=black] (18,3) circle (.2)
node[label=above:{$x_1$}] {};
\fill[fill=white,draw=black] (21,0) circle (.2)
node[label=below:$x_4$] {};
\fill[fill=white,draw=black] (21,3) circle (.2)
node[label=above:$x_2$] {};
\fill[fill=white,draw=black] (24,0) circle (.2)
node[label=below:$y_4$] {};
\fill[fill=white,draw=black] (24,3) circle (.2)
node[label=above:$y_2$] {};
\end{tikzpicture}
    \caption{A graph $G$ and its whiskered graph
    $w(G)$}
    \label{fig:whisker}
\end{figure}

Whiskered graphs have a number of nice properties,
some of which can be found in
\cites{DE2009,FH2008,V1990,W2009}. In particular, the operation of whiskering always produces a Cohen-Macaulay graph; in other words no matter what $G$ we choose, $I(w(G))$ is a Cohen-Macaulay ideal. This implies also that $I(w(G))$ is an unmixed ideal, or in the language of our current paper, it has a pure independence complex (see, for example, \cite{DE2009}*{Theorem 4.4} for a proof). 
Because the independence complex of a whiskered graph is pure,
the graded Artinian algebra $A(w(G))$ is also level
(see \cite{B1994}). We record these facts in a lemma.

\begin{lemma}\label{lem.whiskerprop} 
If $G$ is a whiskered graph on $2n$ vertices, then the independence complex $\Ind(G)$ is a pure simplicial complex of dimension $n-1$.
Furthermore, the algebra $A(G)$ is level with socle degree
$n$.
 \end{lemma}

\section{Main Result}

In this section we prove our main result. 
As we shall show, \cref{maintheorem1} is
in fact a corollary about the structure of the 
independence complex of a whiskered graph.  In 
particular, we prove that these independence
complexes are all pseudo-manifolds.

\begin{theorem}\label{thm.mainthm}
If $G$ is a whiskered graph on $2n$ vertices, then $\Ind(G)$ is a pseudo-manifold. Moreover, if $|E(G)| \ge n+1$, then 
$\Ind(G)$ is a pseudo-manifold with boundary.
\end{theorem}

\begin{proof}  Let $G=w(H)$, where
$$V(H) = \{x_1, \dots, x_n\} \qand V(G) = V(H) \cup \{y_1, \dots, y_n \}$$ are the vertex sets of $H$ and $G$, respectively.  
We need to verify \cref{defn.pseudo} (1)-(3). By \cref{lem.whiskerprop}, $\Ind(G)$ is a pure $(n-1)$-dimensional simplicial  complex, which  verifies \cref{defn.pseudo} (1). We next prove the remaining two conditions. 

We first show that every $(n - 2)$-face of $\Ind(G)$ is contained in at most two facets.   Observe that if $U$ is an independent set of $G$ and $x_k \notin U$ for some $k\in \{1,\ldots,n\}$, then $U \cup \{y_k\}$  is an independent set of $U$. This is because the vertex $y_k$ connects only to $x_k$. 
We therefore conclude that every $(n - 2)$-face $F$ of $\Ind(G)$ can be written as a disjoint union
$$
F = \{x_i \mid  i \in S\} \cup \{y_j \mid j \in T\}
$$
where $S, T \subset \{1, \dots, n\}$, $S \cap T = \emptyset$, 
and $|S \cup T| = n-1$. In particular, $$F \cup \{x_j\}, \ F \cup \{y_i\} \not \in \Ind(G) 
\qfor j \in T, \  i \in S.$$  
Suppose $\{1,\ldots,n\} \smallsetminus (S\cup T)=\{q\}$. Then there are at most two facets of $\Ind(G)$ that contain $F$, namely 
    \begin{itemize}
      \item $F \cup \{x_q\}$, if $\{x_i \mid i \in S\} \cup \{x_q\}$ is an independent set of $H$, and 
      \item $F \cup \{y_q\}$.
    \end{itemize}
We have now verified \cref{defn.pseudo} (2).

Finally, to show \cref{defn.pseudo} (3) holds, consider
any two facets $F, F' \in \Ind(G)$.  We claim there 
exists a sequence of facets $(F_0, \dots, F_s)$ such that 
$F_0 = F$, $F_s = F'$ and $|F_i \cap F_{i + 1}| = n-1$. 

To prove this claim, we let $Y = \{y_1, \dots, y_n\}$ be the 
facet of $\Ind(G)$ that contains all the vertices of the 
whiskers (which is clearly an independent set of $G$) and let 
$F$ be an arbitrary facet of $\Ind(G)$. After reordering 
the vertices of $\Ind(G)$, $F$ can be written as
$$
F = \{x_1, \dots, x_i, y_{i + 1}, \dots, y_n\}. 
$$
In particular, since $\{x_1,\ldots,x_i\}$ is an independent set of vertices in $G$, every subset of it is also independent. 

Consider the sequence of facets $(Y, F_1, \dots, F_{i - 1}, F_i=F)$, where 
    $$
      F_j = \{x_1, \dots, x_{j}, y_{j + 1}, \dots, y_n\}.
    $$
    Note that the sequence satisfies 
        $$
      |F_j \cap F_{j + 1}| = |\{x_1, \dots, x_j, y_{j + 2}, \dots, y_{n}\}| = n-1.
    $$ 
\cref{defn.pseudo} (3) now holds for any two
facets $F,F'$ of $\Ind(G)$ taking the sequences from
$F$ to $Y$ and $F'$ to $Y$ constructed as above, by gluing the two sequences together, with one of them in reverse order.
Thus, $\Ind(G)$ is a pseudo-manifold.

Note that the condition $|E(G)| \ge n+1$ implies 
that $H$ has at least one edge.  
Consequently, there
exists a maximal independent set $D \subsetneq V(H)$. Up to reordering the indices, we may assume $D = \{x_1, \dots, x_i\}$ for some $i<n$. The $(n - 2)$-face $F = \{x_1, \dots, x_i, y_{i + 1}, \dots, y_{n - 1}\}$ of $\Ind(G)$ is contained in the facet $F \cup \{y_n\}$.  Moreover, since $D$ is a 
maximal independent set, the vertex $x_n$ is adjacent to some $x_j \in D$. In particular, $F$ is only contained in one facet of $\Ind(G)$. Consequently, $\Ind(G)$ is a pseudo-manifold with boundary. 
\end{proof}

To prove Theorem
\ref{maintheorem1}, we need to take into
consideration the
characteristic of the field $\K$.

\begin{corollary}[Characteristic $0$ Case]\label{cor.maincor}
Suppose $G$ is a whiskered graph on $2n$ vertices and at least $n+1$ edges, and 
$A=A(G)$ is the Artinian algebra of $G$ over a field $\K$ of characteristic $0$.  If $\ell$ is the sum of the variables of $A$, then $\times \ell: A_i \to A_{i+1}$ has maximal rank 
when $i<n/2$ or $i=n-1$.
\end{corollary}

\begin{proof}
 Let $G=w(H)$, where
$$V(H) = \{x_1, \dots, x_n\} \qand V(G) = V(H) \cup \{y_1, \dots, y_n \}$$ are the vertex sets of $H$ and $G$, respectively.  Note that $H$ must have at least one edge for $G$ to have $n+1$ edges, so the case $n=1$ is excluded. 
If $H$ has two vertices and one edge, then the whiskering $w(H)$ is the path $P_4$ in four vertices. The Artinian algebra defined by this graph has the WLP by \cite{NT2022}*{Theorem 4.4}.

The map $\times \ell: A_i \to A_{i+1}$ is injective when $i<n/2$ by Lemma \ref{lemma:monomial_level_injective},
making use of the fact that
$A$ is an Artinian level algebra of
socle degree $n$ by Lemma \ref{lem.whiskerprop}.
%
%
%
   Since $|E(G)| \geq n+1$, \cref{thm.mainthm} gives that the simplicial complex $\Ind(G)$ is an $(n - 1)$-dimensional pseudo-manifold with boundary.
    By applying \cref{thm.daonair}, the multiplication map $\times \ell: A_{n - 1} \to A_{n}$ has maximal rank.
\end{proof}

\begin{corollary}[Prime Characteristic $\neq 2$ Case]\label{c:char}
Suppose $G$ is a whiskered graph with $2n$ vertices and at least $n + 1$ edges. Let $A = A(G)$ be the Artinian algebra of $G$, and assume the characteristic of the base field $\K$ is a prime not equal to $2$. If $\ell$ is the sum of the variables of $A$, then $\times \ell: A_i \to A_{i + 1}$ has maximal rank when $i = 1$ or $n - 1$.
\end{corollary}

\begin{proof}
We first consider the multiplication map $\times \ell: A_{n - 1} \to A_n$
for all $n \geq 2$. 
Since for every $m \in A_{n - 1}$ we have
    $$
        \ell m = \sum_{x_i m \neq 0 \text{ in $A$}} x_i m
    $$
 we see that the multiplication map $\times \ell: A_{n - 1} \to A_n$ is represented by a matrix $M$ that only has $0$ and $1$ as entries. Since $\Ind(G)$ is a pseudo-manifold, and since the nonzero monomials of degree $i$ correspond to faces of dimension $i - 1$ of $\Ind(G)$, we know the sum of the entries in each column of $M$ is $1$ or $2$. 
 
  Let $N$ be the matrix obtained from $M$ by multiplying a column by $2$ if the sum of its entries is $1$. Because the characteristic of $\K$ is not $2$, this operation does not affect the rank. By \cite{TH2023}*{Corollary 41} the matrix $N^\top$, and hence also $N$, has full rank if and only if it has full rank in characteristic zero and $N^\top$ has at least as many rows as columns. By \Cref{cor.maincor} $M$, and hence $N$, has full rank when the characteristic of $\K$ is zero, so it remains to verify that $f_{n-2} \ge f_{n-1}$, or equivalently, 
  $\dim_\K A_{n-1} \ge \dim_\K A_n$.  
  This follows from the fact that $\Ind(G)$ is connected, and every $(n-1)$-face contains $n$ faces of dimension $(n - 2)$.

We now consider the multiplication map $\times \ell: A_{1} \to A_2$
for all $n \geq 2$.
Because the case $n=2$ is covered
by the previous case, we assume
$n \geq 3$. As shown in the proof
of \cref{cor.maincor}, we
have $f_0 \leq f_1$ for $\Ind(G)$.  By \cite{TH2023}*{Theorem 54} the map $\times \ell: A_1 \to A_2$ has full rank if it does so when the characteristic of $\K$ is zero. This is indeed the case by \Cref{cor.maincor}.
\end{proof}

Next we give an example demonstrating  why we need to avoid characteristic $2$ in \Cref{c:char}. In general, each prime characteristic $p$ results in the failure of injectivity in the maps going to degrees divisible by $p$, see ~\cite[Section~5]{TH2023}. 

\begin{example}  
    Let $H$ be a complete graph on $3$ vertices with edges $\{x_1,x_2\}$, $\{x_1,x_3\}$, and $\{x_2,x_3\}$. Let $G = w(H)$, $A = A(G)$ the Artinian algebra of $G$ over a field $\K$ of characteristic $2$ and $\ell$ the sum of the variables. Then the map $\times \ell: A_1 \to A_2$ is represented by the following matrix:

$$
    \begin{blockarray}{ccccccc}
        & x_1 & x_2 & x_3 & y_1 & y_2 & y_3 \\
      \begin{block}{c(cccccc)}
        x_1y_2 & 1 & 0 & 0 & 0 & 1 & 0 \\
        x_1y_3 & 1 & 0 & 0 & 0 & 0 & 1 \\
        x_2y_1 & 0 & 1 & 0 & 1 & 0 & 0 \\
        x_2y_3 & 0 & 1 & 0 & 0 & 0 & 1 \\
        x_3y_1 & 0 & 0 & 1 & 1 & 0 & 0 \\
        x_3y_2 & 0 & 0 & 1 & 0 & 1 & 0 \\
        y_1y_2 & 0 & 0 & 0 & 1 & 1 & 0 \\
        y_1y_3 & 0 & 0 & 0 & 1 & 0 & 1 \\
        y_2y_3 & 0 & 0 & 0 & 0 & 1 & 1 \\
      \end{block}
    \end{blockarray}.
$$
Notice that the sum of every row is $0$, since char $\K = 2$. Adding the first $5$ columns to the last column gives us the matrix
$$
    \begin{pmatrix}
      1&0&0&0&1&0\\
      1&0&0&0&0&0\\
      0&1&0&1&0&0\\
      0&1&0&0&0&0\\
      0&0&1&1&0&0\\
      0&0&1&0&1&0\\
      0&0&0&1&1&0\\
      0&0&0&1&0&0\\
      0&0&0&0&1&0
    \end{pmatrix}
$$
so the map does not have full rank when char $\K = 2$.
\end{example}

The example below shows that the hypothesis $|E(G)| \geq n+1$ cannot be dropped from the previous statements.

\begin{example}
Let $H$ be the graph consisting of two isolated vertices
$V = \{x_1,x_2\}$. The whiskered graph $G = w(H)$ has two disjoint edges, namely $E(G) = \{\{x_1,y_1\},\{x_2,y_2\}\}$,
but does not satisfy $|E(G)| \geq  
2+1$.
The facets of $\Ind(G)$ are 
$$\{\{x_1,x_2\}, \{x_1,y_2\},\{y_1,x_2\},\{y_1,y_2\}\}.$$
While $\Ind(G)$ is a pseudo-manifold, it does not have a boundary since every face of dimension
$\dim(\Ind(G))-1 = 0$ occurs in two facets.  So
the hypothesis $|E(G)| \geq n+1$ in  \cref{thm.mainthm} cannot
be dropped.  Furthermore, using {\it Macaulay2} \cite{M2},
it can be shown that $\times \ell: A(G)_1 \rightarrow A(G)_2$  does not have maximal rank, so we also need
$|E(G)| \geq n+1$ in \cref{cor.maincor}.
\end{example}

As we will see in the next section, we cannot
improve \cref{cor.maincor} to degrees $2$ and
$n-2$ for all whiskered graphs.


\section{Illustrative examples}

We conclude this note with some illustrative examples.
These examples show that some of the conclusions of \cref{cor.maincor} cannot
be improved for whiskered graphs.  Moreover, they also show that for some natural families of graphs one may wish to consider, the WLP fails.

We first define a broom graph.  Let $m \geq 1$ be an
integer.  The {\it broom graph} $B_m$ is a graph
on $m+3$ vertices $V(G) = \{x_1,x_2,x_3,x_4,\ldots,x_{m+3}\}$ with edge set
$$E(G) = \{\{x_1,x_2\},\{x_2,x_3\}\} \cup \{\{x_3,x_i\}~
|~ i \in {4,\ldots,m+3}\}.$$
Note that our definition of a broom graph is similar
to that \cites{BHO2011,BHO2012} which was defined for
directed graphs.   Our main example is based upon
whiskering a broom graph.

\begin{example}
We consider the broom graph $B_5$ and its 
whiskered graph $G=w(B_5)$;  these graphs are shown in \cref{fig:broom}.   Using {\it Macaulay2}~\cite{M2} to compute the Hilbert function
of the graded Artinian ring $A=A(G)$ we obtain:
$$\begin{array}{c|cccccccccccc} 
i & 0 & 1 & 2 &3 &4 & 5 & 6 & 7 &8 & 9 & \\
\hline
H_{A}(i) &1&16 & 105& 380& 840& 1167& 996& 477& 98 & 0
\end{array}$$

Since $B_5$ is a graph on $|V(B_5)|=8$ vertices, 
\cref{cor.maincor} implies that the map
$\times \ell:A_7 \rightarrow A_8$ has maximal rank (where $\ell$ is the sum of the variables of $A$).  Moreover, the ranks of the maps are:

$$\begin{array}{c|ccccccccccc} 
\times \ell  & 0 \to 1& 1 \to 2& 2 \to 3 &3 \to 4 &4 \to 5& 5 \to 6& 6 \to 7& 7 \to 8\\
\hline
\mbox{rank} &1& 16& 105& 380& 826& 922& 475& 98 &\\
\hline
\mbox{full rank?} &\mbox{yes}& \mbox{yes}& \mbox{yes}& \mbox{yes}& \mbox{no}& \mbox{no}& \mbox{no}& \mbox{yes.} &
\end{array}$$

Note that this example shows that we cannot improve \cref{cor.maincor} to show that the map $\times \ell:A(G)_{i} \rightarrow A(G)_{i+1}$, for $i=\lceil n/2 \rceil$ or $n-2$, has maximal rank for all whiskered graphs $G$ with $2n$ vertices.  Indeed, the above table 
shows that  the map $\times \ell:A_6 \rightarrow A_7$ fails to be surjective, and that $\times \ell: A_4 \to A_5$ fails to be injective. 
\end{example}

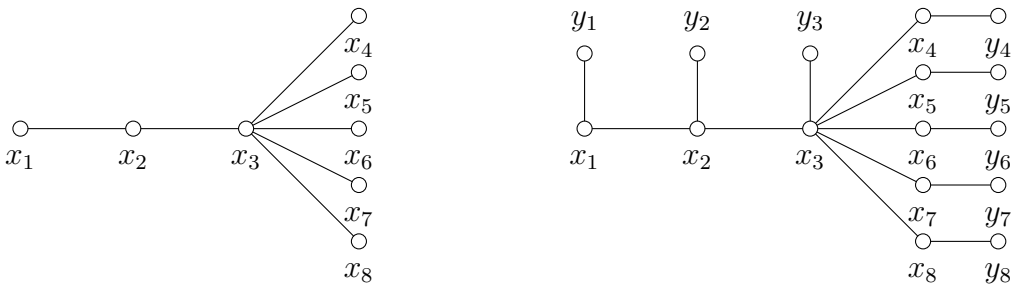
\begin{figure}[h]
\begin{tikzpicture}[scale=0.5]

\draw (0,0) -- (3,0) -- (6,0) -- (9,3);
\draw (6,0) -- (9,1.5);
\draw (6,0) -- (9,0);
\draw (6,0) -- (9,-1.5);
\draw (6,0) -- (9,-3);

\draw (15,0) -- (18,0) -- (21,0) -- (24,3);
\draw (21,0) -- (24,1.5);
\draw (21,0) -- (24,0);
\draw (21,0) -- (24,-1.5);
\draw (21,0) -- (24,-3);

\draw (15,0) -- (15,2);
\draw (18,0) -- (18,2);
\draw (21,0) -- (21,2);
\draw (24,3) -- (26,3);
\draw (24,1.5) -- (26,1.5);
\draw (24,0) -- (26,0);
\draw (24,-1.5) -- (26,-1.5);
\draw (24,-3) -- (26,-3);


\fill[fill=white,draw=black] (0,0) circle (.2)
node[label=below:$x_1$] {};
\fill[fill=white,draw=black] (3,0) circle (.2)
node[label=below:$x_2$] {};
\fill[fill=white,draw=black] (6,0) circle (.2)
node[label=below:$x_3$] {};
\fill[fill=white,draw=black] (9,3) circle (.2)
node[label=below:$x_4$] {};
\fill[fill=white,draw=black] (9,1.5) circle (.2)
node[label=below:$x_5$] {};
\fill[fill=white,draw=black] (9,0) circle (.2)
node[label=below:$x_6$] {};
\fill[fill=white,draw=black] (9,-1.5) circle (.2)
node[label=below:$x_7$] {};
\fill[fill=white,draw=black] (9,-3) circle (.2)
node[label=below:$x_8$] {};

\fill[fill=white,draw=black] (15,0) circle (.2)
node[label=below:$x_1$] {};
\fill[fill=white,draw=black] (18,0) circle (.2)
node[label=below:$x_2$] {};
\fill[fill=white,draw=black] (21,0) circle (.2)
node[label=below:$x_3$] {};

\fill[fill=white,draw=black] (15,2) circle (.2)
node[label=above:$y_1$] {};
\fill[fill=white,draw=black] (18,2) circle (.2)
node[label=above:$y_2$] {};
\fill[fill=white,draw=black] (21,2) circle (.2)
node[label=above:$y_3$] {};
\fill[fill=white,draw=black] (24,3) circle (.2)
node[label=below:$x_4$] {};
\fill[fill=white,draw=black] (24,1.5) circle (.2)
node[label=below:$x_5$] {};
\fill[fill=white,draw=black] (24,0) circle (.2)
node[label=below:$x_6$] {};
\fill[fill=white,draw=black] (24,-1.5) circle (.2)
node[label=below:$x_7$] {};
\fill[fill=white,draw=black] (24,-3) circle (.2)
node[label=below:$x_8$] {};
\fill[fill=white,draw=black] (26,3) circle (.2)
node[label=below:$y_4$] {};
\fill[fill=white,draw=black] (26,1.5) circle (.2)
node[label=below:$y_5$] {};
\fill[fill=white,draw=black] (26,0) circle (.2)
node[label=below:$y_6$] {};
\fill[fill=white,draw=black] (26,-1.5) circle (.2)
node[label=below:$y_7$] {};
\fill[fill=white,draw=black] (26,-3) circle (.2)
node[label=below:$y_8$] {};

\end{tikzpicture}
    \caption{The broom graph
    $B_5$ and its whiskered graph
    $w(B_5)$}
    \label{fig:broom}
\end{figure}

\begin{example}
The broom graph $B_1$ is the same as the path
on four vertices.  The Hilbert function
of the graded Artinian ring $A=A(w(B_1))$ is given
by
$$\begin{array}{c|cccccc} 
i & 0 & 1 & 2 & 3 & 4 & 5  \\
\hline
H_{A}(i) &1&8 & 21& 22& 8 & 0
\end{array}$$
 The map $\times \ell:A_2 \rightarrow A_3$ is not covered by Corollary \ref{cor.maincor}, and we have verified by computer assisted computations that this map indeed does not have maximal rank. 
\end{example}

    Recall that our motivating problem, \cref{question}, asks if there are 
    families of graphs for which $A(G)$ always
    have the weak Lefschetz property.  Our previous
    example shows that such a family of graphs cannot be
    all bipartite graphs or all chordal graphs (a graph where the only induced cycles are triangles), since
    $w(B_5)$ is both a bipartite and chordal graph. Even the family of unmixed trees does not always satisfy the WLP.  In addition, combinatorial commutative algebra has introduced families of graphs whose independence 
    complexes have especially nice combinatorial
    properties, e.g., Cohen-Macaulay, shellable,
    vertex decomposable (see \cite{MV2012} for details 
    and definitions). Since all whiskered graphs 
    are vertex decomposable by
    \cite{DE2009}*{Theorem 4.4} (and thus shellable
    and Cohen-Macaulay), our previous example also shows
    that these families of graphs do not always satisfy the
    weak Lefschetz property.

\begin{remark}
    From computations in Macaulay2 for $m=1,\ldots,8$, 
    the map 
    $$\times \ell:A(w(B_m))_{m+1} \rightarrow A(w(B_m))_{m+2}$$ always fails maximal rank due to the element $y_1y_3 y_4\dots y_{m + 3} \in A(w(B_m))_{m+2}$  not being in the image of the multiplication map. 
We conjecture that this is the case for all $m\geq 1$, and we are currently working on verifying this observation.
\end{remark}

Lastly, we give an example of a whiskered graph that satisfies the weak Lefschetz property. The example below is a whiskering of a complete graph on five vertices. Computational evidence suggests that whiskerings of all complete graphs satisfy WLP. We are working on a proof of this statement.

\begin{example}
    Let $\K$ be a field of characteristic $0$ and $G = w(K_5)$, as shown in
    Figure \ref{fig:whiskerk5}.
\begin{figure}[h]
    \begin{tikzpicture}[scale=0.5]

    \draw (18,0) -- (18,3) -- (21,3) -- (21,0) -- (18,0) -- (21,3);
    \draw (15,0) -- (24,0);
    \draw (15,3) -- (24,3);
    \draw (18,3) -- (21, 0);
    \draw (19.5,4.5) -- (22.5,4.5);
    \draw (18,3) -- (19.5, 4.5);
    \draw (19.5, 4.5) -- (21, 3);
    \draw (21, 3) -- (19.5, 4.5);
    \draw (18, 0) -- (19.5, 4.5);
    \draw (21, 0) -- (19.5, 4.5);
    
    \fill[fill=white,draw=black] (15,3) circle (.2)
    node[label=above:$y_1$] {};    
    \fill[fill=white,draw=black] (15,0) circle (.2)
    node[label=below:$y_2$] {};
    \fill[fill=white,draw=black] (24,0) circle (.2)
    node[label=below:$y_3$] {};
    \fill[fill=white,draw=black] (24,3) circle (.2)
    node[label=above:$y_4$] {};
    \fill[fill=white,draw=black] (22.5,4.5) circle (.2)
    node[label=above:{$y_5$}] {};

    \fill[fill=white,draw=black] (18,3) circle (.2)
    node[label=above:{$x_1$}] {};
    \fill[fill=white,draw=black] (18,0) circle (.2)
    node[label=below:{$x_2$}] {};
    \fill[fill=white,draw=black] (21,0) circle (.2)
    node[label=below:$x_3$] {};
    \fill[fill=white,draw=black] (21,3) circle (.2)
    node[label=above:$x_4$] {};
    \fill[fill=white,draw=black] (19.5,4.5) circle (.2)
    node[label=above:{$x_5$}] {};
    
    \end{tikzpicture}
        \caption{The graph $G = w(K_5)$}
        \label{fig:whiskerk5}
    \end{figure}
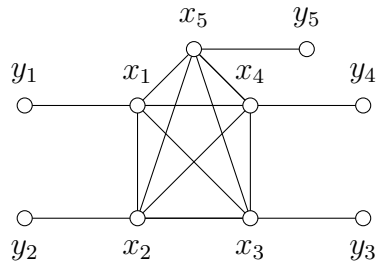
    Using Macaulay2, it can be shown that $A(G)$ satisfies the WLP.
\end{example}

\begin{remark} 
The whiskering construction of graphs has been generalized in several different ways by Biermann--Van Tuyl \cite{BVT2013}, 
Cook--Nagel \cite{CN2012}, and Faridi \cite{Fa2005}. 
We are currently exploring (\cite{grafting}) if these constructions also allow us to construct simplicial complexes 
that are pseudo-manifolds (perhaps with boundary), thus
providing a deeper generalization \cref{thm.mainthm}.
\end{remark}

\noindent
{\bf Acknowledgments.}
Work on this project began at the
\say{Workshop on Lefschetz Properties in Algebra, Geometry, Topology and Combinatorics}, held at The Fields Institute for Research in Mathematical Sciences, Toronto, Canada in May 2023.
We thank The Fields Institute for its hospitality, Eran Nevo  for insightful conversations, and Satoshi Murai for pointing us to Hausel's lemma.  We also thank the anonymous referee for their useful comments.  
 Cooper's research is supported by NSERC Discovery Grant 2018-05004.
Faridi's research is supported by
NSERC Discovery Grant 2023-05929.
Nicklasson's research is supported by the grant KAW-2019.0512.
Van Tuyl’s research is supported by NSERC Discovery Grant 2019-05412.  

\end{document}